\documentclass{amsart}

\usepackage{graphicx}
\usepackage{float}
\usepackage{url}
\usepackage[style=alphabetic]{biblatex}
\usepackage{multirow}
\usepackage{array}

\usepackage{booktabs}
\usepackage{amsfonts}
\usepackage{amsthm}
\usepackage{amsmath}
\usepackage{outlines}
\usepackage{amssymb}
\usepackage[top=1in,bottom=1in,left=1in,right=1in,twoside=false]{geometry}
\usepackage[colorlinks=true,linkcolor=blue]{hyperref}
\usepackage{algpseudocode}
\usepackage{placeins}
\usepackage{hyphenat}
\usepackage{indentfirst}
\usepackage{mathtools}
\usepackage{multicol}
\usepackage{enumitem}
\usepackage[T2A]{fontenc}

\newcolumntype{C}[1]{>{\centering\arraybackslash}p{#1}}

\newtheorem{theorem}{Theorem}[section]
\newtheorem{lemma}[theorem]{Lemma}
\newtheorem{corollary}[theorem]{Corollary}

\newtheorem{proposition}[theorem]{Proposition}
\newtheorem{conjecture}[theorem]{Conjecture}
\theoremstyle{definition}

\newtheorem{remark}[theorem]{Remark}

\title[Hyperelliptics over Characteristic 2 Fields]{On Weil Polynomials of Hyperelliptic Curves  over Finite Fields of Characteristic 2}
\author{Matvey~Borodin}
\address{Department of Mathematics, Massachusetts Institute of Technology, 182 Memorial Dr, Cambridge, MA 02139}
\email{matveyborodin1@gmail.com}

\author{Liam~May}
\address{Department of Mathematics, Massachusetts Institute of Technology, 182 Memorial Dr, Cambridge, MA 02139}
\email{liammay2@mit.edu}

\thanks{The first author was supported by the UROP fund for First-Year and Sophomore students.}
\thanks{The second author was supported by the Paul E. Gray (1954) UROP Fund.}

\date{August 2025}

\newcommand{\F}{\mathbb{F}}
\newcommand{\Z}{\mathbb{Z}}
\newcommand{\R}{\mathbb{R}}
\renewcommand{\P}{\mathbb{P}}
\renewcommand{\O}{\mathcal{O}}

\DeclareMathOperator{\GL}{GL}
\DeclareMathOperator{\PGL}{PGL}
\DeclareMathOperator{\Stab}{Stab}

\bibliography{bibliography}

\begin{document}

\begin{abstract} 
    We present new conditions which obstruct the existence of hyperelliptic Jacobians in isogeny classes of abelian varieties over finite fields of characteristic 2. We show that Weil polynomials of Jacobians cannot have coefficients in certain residue classes modulo 2, extending the approach of~\cite{costa2020}. We prove that for 3- and 4-dimensional abelian varieties over $\F_{2^n}$, as $n \rightarrow\infty$, the parities of the Weil coefficients asymptotically equidistribute. Further, we show that these obstructions disqualify $\frac12$ of all 3-dimensional isogeny classes and $\frac 58$ of all 4-dimensional isogeny classes from containing a hyperelliptic Jacobian. Additionally, we present a practical enumeration algorithm which generates all isomorphism classes of hyperelliptic curves of arbitrary genus over almost any finite field of characteristic 2 based on existing algorithms over $\F_2$. Our analysis shows the runtime to be $\tilde{\O}(2^{n(2g-1)})$ expected, and $\tilde{\O}(2^{n(2g+2)})$ worst case. This runtime improvement renders the algorithm practical for fields other than $\F_2$.

\end{abstract}

\maketitle

\section{Introduction}
The motivation for this paper is to determine which isogeny classes of abelian varieties over a finite field $\F_q$ contain the Jacobian of a hyperelliptic curve. It is a well-known consequence of the Honda-Tate theorem \cite{HONDA, tate} that every isogeny class of abelian varieties of dimension $g$ over a finite field can be uniquely identified with the characteristic polynomial of its Frobenius endomorphism. All such characteristic polynomials are Weil polynomials, which are degree $2g$ monic polynomials over $\Z$ with roots of modulus $\sqrt{q}$. Every such polynomial must have the form $x^{2g} + a_1x^{2g-1} + \ldots+ a_{g-1}x^{g+1}+a_gx^g + a_{g-1}qx^{g-1}+\ldots+a_1q^{g-1}x + q^g$. 

The question of which isogeny classes contain genus $g$ hyperelliptic Jacobians is answered in \cite{WATERHOUSE1969} for genus 1, and has been well studied in genus 2, with a complete classification presented by~\citeauthor{Howe2009} in \cite{Howe2009}. Far less is currently known about the case of genus 3, though there are a number of known conditions on the coefficients of a Weil polynomial that disprove the existence of a hyperelliptic Jacobian in the associated isogeny class. In this paper we refer to such conditions colloquially as obstructions. We provide a detailed survey of known and conjectured obstructions applicable to the genus 3 case of this problem in \cite{borodinmaygeneral}.

In this paper, we restrict our attention to the case of hyperelliptic curves over finite fields of characteristic 2. One of the most powerful known obstructions for isogeny classes of genus 3 hyperelliptic Jacobians over fields of odd characteristic is proved in \cite{costa2020} by \citeauthor{costa2020} who show that an isogeny class of abelian threefolds with Weil polynomial
\[x^6 + sx^5 + tx^4 + ux^3 + tqx^2 + sq^2x + q^3\]
cannot contain the Jacobian of a hyperelliptic curve over $\F_q$ (with $q$ odd) if $t \equiv 0 \pmod{2}$ and $u \equiv 1 \pmod{2}$. They also extend this condition to larger genera, though they don't provide a closed form of the condition.

In Section~\ref{sec:rest_mod_2}, we prove an analogous statement for genus 3 hyperelliptics over finite fields of characteristic 2:
\newtheorem*{thm:rest_mod_2}{Theorem \ref{thm:rest_mod_2}}
\begin{thm:rest_mod_2}
    Let $q$ be a power of 2. Then the isogeny class with Weil polynomial 
    $$x^6 + sx^5 + tx^4 + ux^3 + tqx^2 + sq^2x + q^3$$
    does not contain the Jacobian of a hyperelliptic curve in the following cases:
    \begin{enumerate}
        \item $(s, t, u) = (0, 1, 1) \pmod{2}$
        \item $(s, t, u) = (1, 0, 1) \pmod{2}$
    \end{enumerate}
\end{thm:rest_mod_2}
We conjecture that $(s, t, u) = (1, 1, 0) \pmod{2}$ is also an obstruction in Conjecture~\ref{conj:rest_mod_2} based on extensive computational evidence. 

In Theorem~\ref{thm:rest_genus4}, we prove a similar claim for Weil polynomials of 4-dimensional abelian varieties, leaving one case as Conjecture~\ref{conj:g4_rest}. We provide a general procedure for generating a list of similar parity obstructions for arbitrary genus, and remark on the asymptotic applicability of these rules.

In Section~\ref{sec:asymptotics}, we prove that as $q \to \infty$, Theorem~\ref{thm:rest_mod_2} applies to exactly $\frac 12$ of isogeny classes of 3-dimensional abelian varieties in Theorem~\ref{thm:ratio}, and argue based on computational data that Theorem~\ref{thm:rest_mod_2} identifies 100\% of isogeny classes of abelian threefolds without Jacobians of hyperelliptic curves as $q \to \infty$. Additionally, we prove that as $q\to\infty$, Theorem~\ref{thm:rest_genus4} applies to exactly $\frac 58$ of all isogeny classes of 4-dimensional abelian varieties over finite fields of characteristic 2 in Theorem~\ref{thm:ratiog4}. 

Our work utilized the abelian varieties data available on the L-functions and Modular Forms Database \cite{lmfdb} (see \cite{dupuy2020isogenyclassesabelianvarieties} for a detailed overview of the database). The database contains the data of all 3- and 4-dimensional abelian varieties over finite fields of size up to 25 and 5, respectively. In attempting to compute data over larger fields, we utilized an existing algorithm for generating all genus 3 hyperelliptic curves over a given finite field of odd characteristic (introduced in \cite{howe2024enumeratinghyperellipticcurvesfinite}), but this method fails for finite fields of characteristic 2. \citeauthor{xarles2020censusgenus4curves} proposes an algorithm in \cite{xarles2020censusgenus4curves} to enumerate curves over $\F_2$, which was generalized to genus 5 \cite{dragutinović2022computingbinarycurvesgenus} and genus 6 \cite{huang2024censusgenus6curves}. With some modifications, the same principles can be applied over arbitrary finite fields of characteristic 2, but the runtime of the implementation given scales with $q^{12}$ for genus 3 in the number of SageMath polynomial operations, which becomes impractical for fields of size $\geq 4$.
In Section~\ref{sec:algorithm} we provide a new algorithm to enumerate hyperelliptic curves of arbitrary genus over all fields $\F_{2^n}$ with $\gcd(g+1, 2^n - 1) = 1$ and an alternate implementation which significantly reduces the runtime (allowing us to compute genus 3 and genus 4 hyperelliptics over fields of size up to $\F_{32}$ and $\F_8,$ respectively). We show the expected runtime to be quasi-linear in the size of the output and use our asymptotic bounds to give an upper bound on the number of hyperelliptics in arbitrary genus $g$. In practice, the algorithm appears to perform close to the expected runtime.

An implementation of our algorithm can be found in

\begin{center}
\url{https://github.com/bmatvey/char_2_hyperelliptics}. 
\end{center}

\section*{Acknowledgments}

First and foremost, we would like to thank our wonderful mentor, David Roe, for suggesting this problem and supporting us throughout our endeavors with invaluable advice. We would also like to thank Edgar Costa and Stefano Marseglia for taking the time to discuss our work with us and offer advice. Finally, we thank Everett Howe and Professor Kiran Kedlaya for their help. The first author was supported by the UROP Fund for First-Year and Sophomore Students and the second author was supported by the Paul E. Gray (1954) UROP Fund while working on this project. 

\section{Coefficient Parity Obstructions}\label{sec:rest_mod_2}

In \cite[Thm 2.8]{costa2020}, \citeauthor{costa2020} prove that for an isogeny class of 3-dimensional abelian varieties over a finite field $\F_q$ of odd characteristic, if the Weil polynomial $x^6 + sx^5 + tx^4 + ux^3 + tqx^2 + sq^2x + q^3$ satisfies $t \equiv 0 \pmod{2}$ and $u \equiv 1 \pmod{2}$, the isogeny class cannot contain the Jacobian of a genus 3 hyperelliptic curve. We generalize the techniques used by \citeauthor{costa2020} to finite fields of characteristic 2, proving a new pair of obstructions for genus 3 hyperelliptics. Additionally, we describe a simple technique that allows for the computation of such obstructions for arbitrary genus.

\subsection{Preliminaries}
We first establish necessary preliminaries. These are analogues of statements in \cite[\S 4]{costa2020}, with proofs adapted to fields of characteristic 2. 

Consider a hyperelliptic curve $C/\F_q$ where $\F_q$ has characteristic 2. Let $\pi: C \to \P^1$ be the canonical degree 2 cover of the projective line by $C$ given by $\pi: (x, y) \to x$. Define $W$ as the support of the ramification divisor of $\pi$; thus, $W$ is the set of geometric Weierstrass points of $C$. Since the field has characteristic 2, $\#W=r+1$, where $r$ is the 2-rank of $C$ (this follows from the Deuring-Shafarevich formula \cite[Cor. 1.8]{crew_etale}). The Weierstrass points are stable under the Frobenius endomorphism $\varphi$, which acts on $W$ by permutation. In particular, the action of $\varphi$ partitions the points $w_i \in W$ into orbits; let us denote the cardinalities of the orbits by $\{d_i\}$. Since $|W| = r + 1$, we have $\sum d_i = r + 1$.  

\begin{lemma} \label{lemma:costa41}
    For each $n \geq 1$, we have $\#W(\F_{q^n}) = \sum_{i: d_i \mid n} d_i$.
\end{lemma}

\begin{proof}
    Given a Weierstrass point $w_i \in W$ with orbit length $d_i$, the extensions of $\F_q$ in which $w_i$ is defined are exactly the fields $\F_{q^{md_i}}$ for integers $m$. Thus, in each field $\F_{q^n}$, we see $d_i$ Weierstrass points for each $d_i \mid n$ and 0 Weierstrass points for every other $d_i$.
\end{proof}

\begin{lemma}\label{lemma:costa42}
    For each $n \geq 1$, we have $\#C(\F_{q^n}) \equiv \#W(\F_{q^n}) \pmod{2}$.
\end{lemma}

\begin{proof}
    Consider the hyperelliptic involution $\iota: C \to C$. For every $\F_{q^n}$-rational point $P$ on the curve $C$ over $\F_{q^n}$, $\iota(P)$ is also a rational point on $C$. Moreover, the Weierstrass points are exactly the fixed points of $\iota$. Thus, every non-Weierstrass point contributes 2 points to $\#C(\F_{q^n})$, while every Weierstrass point contributes 1 point, so $\#C(\F_{q^n}) \equiv \#W(\F_{q^n}) \pmod{2}$. 
\end{proof}

\begin{proposition}\label{prop:costa47}
    Given $n = 2^a \cdot m$ with $a \geq 1$ and $m$ odd, we have 
    \[\#C(\F_{q^n}) \equiv 2(q^n + 1) - \#W(\F_{q^n}) \equiv 2 - \#W(\F_{q^n}) \pmod{2^{a+1}}.\]
\end{proposition}
\begin{proof}
    We use the Weierstrass form for $C$, that is $y^2 + h(x)y = f(x)$ for polynomials $h(x), f(x) \in \F_q[x]$. Consider the fibers of the hyperelliptic map $\pi: C(\F_{q^n}) \to \P^1(\F_{q^n})$. Any point $x \in \P^1(\F_{q^n})$ must have either 0, 1, or 2 preimages over $\F_{q^n}$. Notably, $\pi^{-1}(x)$ has one point if and only if $x$ is the image of a Weierstrass point that is defined over $\F_{q^n}$. Now, we consider $x \not\in W$. If $x \in \F_{q^{n/2}}$ then $y^2 + h(x)y - f(x) \in \F_{q^{n/2}}[y]$, so it has two roots over $\F_{q^n}$, and therefore $x$ has two preimages on $C(\F_{q^n})$. Conversely, if $x \in \P^1(\F_{q^n})$ is not defined over $\F_{q^{n/2}}$, the size of the Galois orbit of $x$ is divisible by $2^a$. Now, $x$ must have either 0 or 2 preimages in $C(\F_{q^n})$. If $x$ has 0 preimages in $C(\F_{q^n})$, $x$ must have two preimages in $C(\F_{q^{2n}})$ which belong to a Galois orbit of size divisible by $2^{a+1}$. Thus, the number of contributed points modulo $2^{a+1}$ is the same as the case in which $x$ has two preimages over $\F_{q^n}$. Therefore, to count $\#C(\F_{q^n})$ modulo $2^{a+1}$, we assume all non-Weierstrass points have 2 preimages on the curve, and all Weierstrass points have 1 preimage. Finally, note that $q^n \equiv 0 \pmod{2^a}$.
\end{proof}

\subsection{Genus 3 Obstructions}
\noindent We prove an analogue of the result in \cite{costa2020} for fields of characteristic 2. We use the following formulae, also presented in \cite[\S 4]{costa2020}. Note that the version used in \cite{costa2020} is missing the term $5s^2u$ in the fifth equation. These utilize Newton's identities and \cite[Thm. 11.1]{milne86}.

\begin{lemma} \label{lemma:point_counting}
    Letting $s$, $t$ and $u$ be the coefficients of a Weil polynomial of an isogeny class containing the Jacobian of a genus 3 hyperelliptic curve $C/\F_q$, we have:
    \begin{equation}
    \begin{split}
        \#C(\F_{q}) &= q + 1 + s\\
        \#C(\F_{q^2}) &= q^2 + 1 - s^2 + 2t\\
        \#C(\F_{q^3}) &= q^3 + 1 + s^3 - 3st + 3u \\
        \#C(\F_{q^4}) &= q^4 + 1 - s^4 + 4s^2t - 4su - 2t^2 + 4qt\\
        \#C(\F_{q^5}) &= q^5 + 1 + s^5 - 5s^3t + 5st^2 + 5s^2 u - 5qst - 5tu + 5q^2 s
    \end{split}
    \end{equation}
\end{lemma}

\begin{theorem}\label{thm:rest_mod_2}
    Let $q$ be an even prime power. Then the isogeny class with Weil polynomial 
    $$x^6 + sx^5 + tx^4 + ux^3 + tqx^2 + sq^2x + q^3$$
    cannot contain a Jacobian of a hyperelliptic curve in the following cases:
    \begin{enumerate}
        \item $(s, t, u) = (0, 1, 1) \pmod{2}$
        \item $(s, t, u) = (1, 0, 1) \pmod{2}$
    \end{enumerate}
\end{theorem}

\begin{proof}
    \textbf{Case 1}: Using Lemma~\ref{lemma:point_counting}, obtain the following point counts:
    \[
        \#C(\F_{q}) \equiv 1, 
        \#C(\F_{q^3}) \equiv 0, 
        \#C(\F_{q^5}) \equiv 0 \pmod{2}
    \]
    By Lemmas~\ref{lemma:costa41}, \ref{lemma:costa42}, the first congruence implies the existence of an odd number of $d_i = 1$. But then the third congruence implies that there exists a $d_i = 5$. From the parity of the coefficients, we know the 2-rank is 3, so $\sum d_i = 4$, leading to a contradiction.
    
    \textbf{Case 2}: Using Lemma~\ref{lemma:point_counting}, obtain the following point counts:
    \[
        \#C(\F_{q}) \equiv 0, 
        \#C(\F_{q^3}) \equiv 1, 
        \#C(\F_{q^5}) \equiv 1 \pmod{2}
    \]
    By Lemmas~\ref{lemma:costa41}, \ref{lemma:costa42}, the first congruence implies the existence of an even number of $d_i = 1$; thus the third congruence implies the existence of at least one $d_i = 5$. We once again have a contradiction to $\sum d_i = 4$ (note that the 2-rank must once more be 3). 
    
\end{proof}

\begin{conjecture}\label{conj:rest_mod_2}
    Let $q$ be an even prime power. Then the isogeny class with Weil polynomial 
    $$x^6 + sx^5 + tx^4 + ux^3 + tqx^2 + sq^2x + q^3$$
    does not contain a Jacobian of a hyperelliptic curve if $(s, t, u) = (1, 1, 0) \pmod{2}$.
\end{conjecture}

We confirmed this conjecture for all Weil polynomials of isogeny classes of abelian varieties over $\F_2, \F_4, \F_8$ and $\F_{16}$. As it does not violate any of the point counting arguments above, the same method of proof does not apply. 

\subsection{Genus 4 Obstructions}

We begin with an analogue of Lemma~\ref{lemma:point_counting}:

\begin{lemma} \label{lemma:newton4}
    Letting $s,t,u,$ and $v$ be the coefficients of the Weil polynomial of the isogeny class containing the Jacobian of the genus 4 hyperelliptic curve $C/\F_q$, we have:
    \begin{enumerate}
        \item $\#C(\F_{q})   = q + 1 + s$
        \item $ \#C(\F_{q^2}) = q^2 + 1 - s^2 + 2t$
        \item $\#C(\F_{q^3}) = q^3 + 1 + s^3 - 3st + 3u $
        \item $ \#C(\F_{q^4}) = q^4 + 1 - s^4 + 4s^2t - 4su - 2t^2 + 4v$
        \item $\#C(\F_{q^5}) = q^5 + 1 + s^5 - 5s^3t + 5s^2 u + 5st^2 - 5sv - 5tu + 5uq$
        \item $ \#C(\F_{q^6}) = q^6 + 1 - s^6 + 6s^4t - 6s^3u - 9s^2t^2 + 6s^2v + 12stu - 6suq+ 2t^3 - 6tv + 6tq^2 - 3u^2$
        \item $ \#C(\F_{q^7}) = q^7 + 1 + s^7 - 7s^5 t + 7s^4 u + 14s^3 t^2 - 7s^3 v - 21s^2 tu  + 7s^2 u q - 7st^3 + 14stv - 7st q^2 
                       + 7su^2 + 7s q^3 + 7t^2 u - 7tu q - 7uv$
    \end{enumerate}
        
\end{lemma}

Using this lemma, we can now prove an analogue of Theorem~\ref{thm:rest_mod_2} for genus 4:

\begin{theorem} \label{thm:rest_genus4}
    Let $f(x) = x^8 + sx^7 + tx^6 + ux^5 + vx^4 + uqx^3 + tq^2x^2 + sq^3x + q^4$ be the Weil polynomial of an isogeny class $\mathcal{A}$ of 4-dimensional abelian varieties over a finite field $\F_q$ of characteristic 2. If the coefficients $(s, t, u, v)$ are equivalent to any of the following modulo 2, $\mathcal{A}$ does not contain a hyperelliptic Jacobian:
    \begin{multicols}{6}
        \begin{enumerate}[leftmargin=*, itemsep=0.25em]
        \item[(I)] $(0,0,1,1)$
        \item[(II)] $(0,1,0,1)$
        \item[(III)] $(1,1,0,1)$
        \item[(IV)] $(0,1,1,1)$
        \item[(V)] $(1,0,0,1)$
        \item[(VI)] $(0,1,1,0)$
        \item[(VII)] $(1,0,1,0)$        
        \end{enumerate}
    \end{multicols}
    
\end{theorem}

\begin{proof}
    \textbf{Case 1:} Using Lemma~\ref{lemma:newton4}, obtain the following point counts:
    \[
        \#C(\F_q) \equiv 1, \#C(\F_{q^3}) \equiv 0 \pmod{2}.
    \]
    Since the $d\equiv1$, the $2$-rank is 4, so we have $\sum d_i$ = 5. By Lemma~\ref{lemma:costa41}, the first congruence implies that there are an odd number of $d_i=1$, and the second congruence thus forces the existence of a $d_i=3$, a contradiction.

    \noindent \textbf{Case 2:} We obtain the following point count:
    \[
        \#C(\F_{q^4}) \equiv 3 \pmod{8}.
    \]
    By Proposition~\ref{prop:costa47}, this implies that $\#W(\F_{q^4}) \equiv 7 \pmod{8}$, a contradiction since $\#W(\F_{q^4}) \leq 5$.
    
    \textbf{Case 3:}  We obtain the following point counts modulo 2:
    \[
        \#C(\F_q)\equiv0, \#C(\F_{q^3}) \equiv 1, \#C(\F_{q^5}) \equiv 1.
    \]
     We have $\sum d_i$ = 5 by the $2$-rank. The first congruence implies an even number of $d_i = 1$, and the second and third congruences imply the existence of both a $d_i=3$ and a $d_i=5$, a contradiction.

    \textbf{Case 4:} This is proven by an identical argument to case 1.

    \textbf{Case 5:} We obtain the following point counts:
    \[
        \#C(\F_q) \equiv 0, \#C(\F_{q^7}) \equiv 1 \pmod{2}.
    \]
    Since $\sum d_i$ = 5 by the $2$-rank, together these congruences imply the existence of a $d_i=7$, a contradiction.

    \textbf{Case 6:} We obtain the following point counts:
    \[
        \#C(\F_q) \equiv 1, \#C(\F_{q^5}) \equiv 0 \pmod{2}.
    \]
    Together these congruences imply the existence of a $d_i=5$, however the $2$-rank is 3, so we have $\sum d_i = 4$, a contradiction.

    \textbf{Case 7:} We obtain the following point counts:
    \[
        \#C(\F_q) \equiv 0, \#C(\F_{q^5}) \equiv 1 \pmod{2}.
    \]
    Together these congruences imply the existence of a $d_i=5$, however the $2$-rank is 3, so we have $\sum d_i = 4$, a contradiction.
    
\end{proof}

\begin{conjecture} \label{conj:g4_rest}
    As in Theorem~\ref{thm:rest_genus4}, if the Weil coefficients $(s, t, u, v)$ associated to $\mathcal{A}$ are equivalent to $(1,1,0,0) \pmod{2}$, $\mathcal{A}$ does not contain a hyperelliptic Jacobian.
\end{conjecture}

\subsection{General Genus} \label{sec:proof_alg}
We now describe a method of proof that can be used to computationally generate a list of conditions which disprove the existence of a genus $g$ hyperelliptic Jacobian in an isogeny class of $g$-dimensional abelian varieties over a finite field of characteristic 2. Let $f(x) = x^{2g} + a_1x^{2g-1} + \ldots + a_gx^g + \ldots + q^g$ be the Weil polynomial of an isogeny class of $g$-dimensional abelian varieties over $\F_q$, and let $\alpha_1,\ldots,\alpha_{2g}$ be the roots over $\mathbb{C}$. For each possible pattern of residues mod 2 in the set $\{0,1\}^g$, say $(n_1,\ldots,n_g)$, compute the $2$-rank $r$ to be the largest index $i$ such that $n_i=1$. Accordingly, the number of Weierstrass points is $\#W=r+1$. The objective is to compute point counts $\#C(\F_{q^\ell})$ for primes $\ell$, and create a contradiction similar to in the proof of Theorem~\ref{thm:rest_mod_2}. This is done recursively as follows. For a polynomial $f$ with a residue pattern $(n_1, \ldots, n_g)$:
\begin{enumerate}
    \item[(I)] Compute the elementary symmetric sums $e_i$ in the Frobenius eigenvalues by reading off the signed coefficients of $f(x)$, obtaining \[e_0=1,e_1 = -a_1, \ldots, e_g=(-1)^ga_g, e_{g+1}=(-1)^{g+1}a_{g-1}q,\ldots,e_{2g}=q^g.\] 
    \item[(II)] Use Newton's identities to recursively compute the power sums $p_k(\alpha_1,\ldots,\alpha_{2g}) = \sum\alpha_i^k$. Explicitly \[p_k = \begin{cases}(-1)^{k-1}ke_k + \sum_{i=1}^{k-1} (-1)^{k-1-i}e_{k-i}p_i\text{ if }k \leq 2g\\
    \sum_{i=k-n}^{k-1}(-1)^{k-1+i}e_{k-i}p_i \text{ if } k > 2g.\end{cases}\]
    \item[(III)] For each $p_k$, we have $\#C(\F_{q^k}) = 1 + q^k - p_k$. Iterate through all possible partitions $\{d_1, d_2, \ldots, d_n\}$ and check if $\sum_{i : d_i \mid k} d_i \equiv \#C(\F_{q^k})\pmod{2}$ for all $k$. If there does not exist a partition satisfying this equivalence for all $k$, we have a contradiction by Lemmas~\ref{lemma:costa41} and \ref{lemma:costa42}. Therefore, we conclude that the residue class $(n_1, \ldots, n_g)$ is an obstruction. 
\end{enumerate}

\begin{remark}
    In step (III) of the procedure above, we may also check conditions of the form  $\#C(\F_{q^k}) \equiv 2 - \sum_{i : d_i \mid k} d_i \pmod{2^{a+1}}$ where $a$ is the 2-adic valuation of $k$, which can provide a contradiction by Proposition~\ref{prop:costa47} (i.e. case II of Theorem~\ref{thm:rest_genus4}). However, it may not be true in general that $\#C(\F_{q^k})$ can be uniquely determined mod $2^{a+1}$ based on the coefficients $(s, t, u, v) \pmod{2}$. Thus, we do not include this condition in the above procedure. 
\end{remark}

\begin{proposition}
    If $(n_1,\ldots,n_g)$ is a modulo 2 obstruction proven by the above procedure and $f(x) = x^{2g} + a_1x^{2g-1} + \ldots + a_gx^g + \ldots + q^g$ is the Weil polynomial of an isogeny class $\mathcal{A}$ satisfying $(a_1,\ldots,a_g) \equiv (n_1,\ldots,n_g) \pmod{2}$, $\mathcal{A}$ cannot contain a hyperelliptic Jacobian.
\end{proposition}

\begin{proof}
    This follows immediately from Lemmas~\ref{lemma:costa41} and \ref{lemma:costa42}.
\end{proof}

\begin{remark}
    To use the above procedure, one must pick some cutoff for how many extensions to check (an upper bound on $k$). In our testing, an upper bound of $2g$ appears sufficient to find almost all obstructions. Additionally, note that while this procedure finds a significant portion of restrictions on the coefficients mod 2, it is not exhaustive. For instance, it is insufficient to produce Conjectures~\ref{conj:rest_mod_2} and~\ref{conj:g4_rest}.
\end{remark}

Further, notice that the last two residues in the hypothesis of Theorem~\ref{thm:rest_genus4} with a zero padded on each are exactly the hypotheses of Theorem~\ref{thm:rest_mod_2}. This connection is true in general, allowing one to lift obstructions from genus $g$ to genus $g+1$:

\begin{proposition}
    If the residue pattern $(n_1, \ldots, n_g)$ is proved to be an obstruction for dimension $g$ Weil polynomials using steps (I)-(III) above, then the residue pattern $(n_1, \ldots, n_g, 0, 0, \ldots)$ with $i$ zeroes is an obstruction for genus $g + i$ Weil polynomials.
\end{proposition}

\begin{proof}
    Since padding with zeroes does not change the 2-rank, the value $\sum d_i = r + 1 = \#W$ does not change. By~\cite[Thm. 1.1, 1.2]{Mosse_2019}, the $k$-th symmetric sum $p_k(x_1,\ldots,x_m)$ in $m$ variables is equal to the $k$-th symmetric sum $p_k(x_1,\ldots,x_m,0,\ldots,0)$ in $m+i$ variables, where the last $i$ variables are set to zero. Since $q \equiv 0 \pmod{2}$, the last $g$ coefficients of a Weil polynomial are 0 mod 2, so increasing the degree does not affect the parity of point counts, producing the same contradiction as by $(n_1,\ldots,n_g).$
\end{proof}

\section{Asymptotics}\label{sec:asymptotics}

In this section, we analyze the applicability of Theorems~\ref{thm:rest_mod_2} and~\ref{thm:rest_genus4} as $q \to \infty$ for $q=2^n$. We begin with a lemma about the distribution of coefficients of Weil polynomials:

\begin{lemma}\label{lemma:even_distribution}
    Let $\mathcal{W}_3(q)$ be the set of all Weil polynomials of the form $x^6 + sx^5 + tx^4 + ux^3 + tqx^2 + sq^2x + q^3$. Then the proportion of elements $\mathcal{W}_3(q)_{(r_1, r_2, r_3)}$ satisfying $(s, t, u) \equiv (r_1, r_2, r_3) \pmod{2}$ for any $r_1, r_2, r_3 \in \Z/2\Z$ is equal to $\frac{1}{8}$ as $q \to \infty$ over prime powers. 
\end{lemma}

\begin{proof}
    By \cite[Main Theorem A]{marseglia2025}, the elements of $\mathcal{W}_3(q)$ are in bijection with triplets of integers $(s, t, u)$ with $s, t, u\in \Z$ satisfying the inequalities
    \[
    \begin{split}
        -6\sqrt{q} \le s &\le 6\sqrt{q}, \\
        4\sqrt{q} |s| - 9q \le t &\le \frac 13 s^2 + 3q, \\
        -\frac{2}{27}s^3 + \frac 13 st + qs - \frac{2}{27}(s^2 - 3t + 9q)^{3/2} \le u &\le -\frac{2}{27}s^3 + \frac 13 st + qs + \frac{2}{27}(s^2 - 3t + 9q)^{3/2}, \\
        -2qs - 2\sqrt{q}t - 2q\sqrt{q} \le u & \le -2qs + 2\sqrt{q}t + 2q\sqrt{q}.
    \end{split}
    \]
    These inequalities define a compact subset $S_3(q) \subset \R^3$, and we are interested in the $\Z^3$ lattice points that fall inside of $S_3(q)$. Consider the partition of $\R^3$ into cubes $C$ with side-length 2 aligned to the integer lattice. Each such cube is either fully contained in $S_3(q)$, disjoint from $S_3(q)$, or intersects both $S_3(q)$ and $\R^3 - S_3(q)$ (we call these interior, exterior, and boundary, respectively). In the first two cases, the vertices of the bottom-left length-1 cube contained in $C$ uniformly contribute either 1 or 0 polynomials, respectively, to every set $\mathcal{W}_3(q)_{(r_1, r_2, r_3)}$. Thus, to prove our result, it is sufficient to show that the number of interior cubes grows faster than the number of boundary cubes. 

    To estimate the cube counts, we consider the change of coordinates $\Psi: (a, b, c) \to (\alpha, \beta, \gamma)$ given by $\alpha = \frac{a}{\sqrt{q}}, \beta = \frac{b}{q}, \gamma = \frac{c}{q^{3/2}}$. Let $T_3(q)$ be the image of $S_3(q)$ under this map. Note that under this change of coordinates, the defining equations of $T_3(q)$ are constant in $q$, so $T_3(q)$ is independent of $q$. Furthermore, $T_3(q)$ is compact, and $\partial(T_3(q))$ is compact, piecewise smooth, and has finite surface area. 

    Since $\Psi$ is invertible, we consider $\Psi^{-1}(T_3(q)) = S_3(q)$. Under $\Psi^{-1}$, volume is scaled by $\det(\Psi^{-1}) = q^3$, and surface area is scaled by a factor of at most $q \cdot q^{3/2} = q^{5/2}$. Therefore, the volume of $S_3(q)$ and the surface area of the boundary are given by $\text{Vol}(S_3(q)) = \Theta(q^3)$ and $\text{Vol}(\partial S_3(q)) = \mathcal{O}(q^{5/2})$ respectively. 

    Now, every cube that intersects $S_3(q)$ is either an interior or boundary cube. We can bound the number of boundary cubes: every boundary cube must lie completely within a 4-neighborhood of $\partial S_3(q)$, which we denote $\mathcal{B}_4(\partial(S_3(q))$ so the number of boundary cubes is bounded above by the volume of this neighborhood divided by 8. Using a tubular neighborhood approximation, the volume is bounded by $\mathcal{O}(q^{5/2})$ so the number of boundary cubes is $\mathcal{O}(q^{5/2})$. 

    Now, we count interior cubes. The volume $S_3(q) - \mathcal{B}_4(\partial(S_3(q))$ can be covered completely by interior cubes, so the number of interior cubes is at least the volume of $S_3(q) - \mathcal{B}_4(\partial(S_3(q))$ divided by 8. This is $\Theta(q^3) - \mathcal{O}(q^{5/2}) = \Theta(q^3)$. We conclude that 
    \[\lim_{q \to \infty} \frac{\mathcal{W}_3(q)_{(r_1, r_2, r_3)}}{\mathcal{W}_3(q)} = \frac{\ell + \mathcal{O}(q^{5/2})}{8\ell + \mathcal{O}(q^{5/2})} = \frac 18,\]
    where $\ell = \Omega(q^3)$ is the number of interior cubes.
\end{proof}

\begin{remark}
    Note that Lemma \ref{lemma:even_distribution} is a result about Weil polynomials, not about Frobenius characteristic polynomials of isogeny classes of abelian varieties. The lemma applies to finite fields of any characteristic, but if we replace ``Weil polynomials'' with ``Frobenius characteristic polynomials of isogeny classes of abelian varieties'', this statement becomes false for characteristic 2 fields, as will become clear from the proof of Theorem~\ref{thm:ratio}. 
\end{remark}

An analogous result to Lemma~\ref{lemma:even_distribution} is true for Weil polynomials of 4-dimensional abelian variates, and the proof is nearly identical. 

\begin{lemma}
    Let $\mathcal{W}_4(q)$ be the set of all Weil polynomials of the form $f(x) = x^8 + sx^7 + tx^6 + ux^5 + vx^4 + uqx^3 + tq^2x^2 + sq^3x + q^4$. Then the proportion of elements $\mathcal{W}_4(q)_{(r_1,r_2,r_3,r_4)}$ satisfying $(s, t, u, v) \equiv (r_1, r_2, r_3, r_4) \pmod{2}$ for any $r_1, r_2, r_3, r_4 \in \Z/2\Z$ is equal to $\frac{1}{16}$ as $q\to\infty$ over prime powers. 
\end{lemma}

\begin{proof}
    This follows from \cite[Main Theorem B]{marseglia2025} using the same method as the proof of Lemma~\ref{lemma:even_distribution}.
\end{proof}

\noindent Now we restrict our attention to irreducible Weil polynomials:

\begin{lemma}\label{lem:irreducible}
    As $q \to \infty$, the proportions $\theta_3(q)$ and $\theta_4(q)$ of elements of $\mathcal{W}_3(q)$ and $\mathcal{W}_4(q)$ which are irreducible over $\Z$ are both 1. 
\end{lemma}

\begin{proof}
    To count the total number of degree $2g$ $q$-Weil polynomials, we are interested in the quantity $\#(\Lambda_q \cap V_g) = {\mathcal{O}}( \frac{\text{Vol
    } (V_g)}{\text{Covol } (\Lambda_q)})$ from \cite[Prop. 2.3.1]{dipippo2000realpolynomialsrootsunit} with $\Lambda_q$ as in \cite[\S 3.2]{dipippo2000realpolynomialsrootsunit}. From the same paper, we obtain $\text{Vol}(V_g) = \mathcal{O}(1)$ for $g=\mathcal{O}(1)$. Using the basis of $\Lambda_q$, we compute \[
    \text{Covol}(\Lambda_q) = \prod_{i=i}^{g} q^{-i/2} = q^{-\sum_{i=1}^n i/2} = q^{-g(g+1)/4},
    \]
    and conclude that $\#(\Lambda_q \cap V_g) = \mathcal{O}(q^{n(n+1)/4}).$
    
    Thus, $\mathcal{W}_1(q) \sim q^{1/2}$, $\mathcal{W}_2(q) \sim q^{3/2}$, $\mathcal{W}_3(q) \sim q^3$, and $\mathcal{W}_4(q) \sim q^5$. Since every reducible Weil polynomial must factor into lower degree Weil polynomials, we obtain the natural bound: \[
    \theta_g(q) \geq \mathcal{W}_g(q)-\prod_{\substack{d_1+\ldots+d_k=g\\k>1}} \mathcal{W}_{d_i}(q),
    \] 
    and computing these values for partitions of $g=3$ and $g=4$, we obtain the desired result.
\end{proof}

\begin{theorem}\label{thm:ratio}
    Let $\tau_3(q)$ denote the fraction of all isogeny classes of 3-dimensional abelian varieties that do not contain the Jacobian of a hyperelliptic curve according to Theorem~\ref{thm:rest_mod_2}. Then 
    \[\lim_{q \to \infty} \tau_3(q) = \frac 12\]
    where the limit is taken over powers of 2. 
\end{theorem}

\begin{proof}
    As a direct consequence of \cite[Cor 4.2.1]{borodinmaygeneral}, we have that as $q \to \infty$, the ratio of the number of ordinary simple isogeny classes of 3-dimensional abelian varieties to the total number of isogeny classes of 3-dimensional abelian varieties is 1. Therefore, since we are concerned with the $q \to \infty$ limit, we only need to consider ordinary simple classes of abelian varieties. In particular, these have $p$-rank 3 in dimension 3, and by \cite[Thm. 1.4]{haloui2010}, irreducible Weil polynomials with $p$-rank 3 are in bijection with the Frobenius characteristic polynomials of ordinary simple isogeny classes of abelian varieties. Applying Lemma~\ref{lem:irreducible}, as $q \to \infty$, all ordinary Weil polynomials are characteristic polynomials of isogeny classes of abelian varieties. Moreover, note that in the case of $p=2$, a Weil polynomial has $p$-rank 3 if and only if $c \equiv 1 \pmod{2}$, so we conclude that the proportion of isogeny classes with Frobenius characteristic polynomials having $c \equiv 1 \pmod{2}$ is 1 as $q \to \infty$.


    Therefore, $q \to \infty$ over powers of 2, exactly $\frac 14$ of all triplets $(a, b, c)$ of defining coefficients of characteristic polynomials of isogeny classes of abelian varieties fall into each of the classes $(0, 0, 1), (0, 1, 1), (1, 0, 1), \\(1, 1, 1)$. Since two of these are an obstruction to a hyperelliptic Jacobian existing in the corresponding isogeny class, $\lim_{q \to \infty} \tau_3(q) = \frac 12$, as desired. 
\end{proof}

To extend this result to genus 4, we begin with a statement about simple isogeny classes in genus 4. 

\begin{lemma} \label{lemma:g4simple}
    As in \cite[Thm 1.1]{dipippo2000realpolynomialsrootsunit}, let us denote the number of isogeny classes of $g$-dimensional abelian varieties over $\F_q$ as \[\mathcal{I}(g, q) = \left(\frac{2^g}{g!} \prod_{i=1}^{g} \left(\frac{2i}{2i - 1}\right)^{g + 1 - i}\right) r(q)q^{g(g+1)/4},\] where $r(q) = \phi(q)/q$ with $\phi(x)$ the Euler function. Let $\mathcal{A}$ be an isogeny class of abelian varieties over $
    F_q$. Write $\mathcal{A} \sim \mathcal{A}_1 \times\ldots\times\mathcal{A}_k$ as a product of simple isogeny classes with $\dim(\mathcal{A}_i)=d_i$, so $d_1+\ldots+d_n=g$ is a partition $P$ of $g$. Define $N_g(P)$ to be the number of $g$-dimensional isogeny classes whose simple factors have dimensions given by $P$. Then $N_4(\{4\})={\mathcal{O}}(\mathcal{I}(4, q))$.
\end{lemma}

\begin{proof}
    We compute the counts for each partition different from $\{4\}$. 
    \begin{equation*}
        \begin{split}
            N_4(\{1,1,1,1\}) &= \binom{\mathcal{I}(1,q)+3}{4} = {\mathcal{O}}(q^2), \\
            N_4(\{1,1,2\}) &= \binom{\mathcal{I}(1,q) + 1}{2}\left(\mathcal{I}(2,q)-\binom{\mathcal{I}(1,q)+1}{2}\right) = {\mathcal{O}}(q(q^\frac32-q)) = {\mathcal{O}}(q^\frac52),\\
            N_4(\{1,3\}) &= \mathcal{I}(1,q)N_3({3}) = {\mathcal{O}}(q^\frac12q^3) = {\mathcal{O}}(q^\frac72), \\
            N_4(\{2,2\}) &= \left(\mathcal{I}(2,q) - \binom{\mathcal{I}(1,q)+1}{2}\right)^2 = {\mathcal{O}}\left((q^\frac32 - q)^2\right) = {\mathcal{O}}(q^\frac32).
        \end{split}
    \end{equation*}
    Note that the quantity $N_3(\{3\})={\mathcal{O}}(q^3)$ is proven in \cite[Lem. 4.4]{borodinmaygeneral}. We conclude by computing \[
    N_4(\{4\}) = \mathcal{I}(4,q) - \sum_{P\neq\{4\}} N_4(P) = {\mathcal{O}}(q^5) - {\mathcal{O}}(q^\frac72) = {\mathcal{O}}(q^5) = \O(\mathcal{I}(4,q)).
    \]
\end{proof}

\begin{theorem}\label{thm:ratiog4}
    Let $\tau_4(q)$ denote the proportion of isogeny classes of 4-dimensional abelian varieties over $\F_q$, where $q=2^n$, which are proven not to contain a hyperelliptic Jacobian by Theorem~\ref{thm:rest_genus4}. Then we have the following limit: \[
    \lim_{q\to\infty} \tau_4(q) = \frac58.
    \] 
\end{theorem}

\begin{proof}
    The proof is analogous to that of Theorem~\ref{thm:ratio}. Using Lemma~\ref{lemma:g4simple} and \cite[Thm 1.2]{dipippo2000realpolynomialsrootsunit}, we see that as $q\to\infty$, the proportion of isogeny classes of 4-dimensional abelian varieties over $\F_q$ which are simple and ordinary approaches 1. Using this fact, we deduce from \cite[Thm 1.2]{haloui2011characteristicpolynomialsabelianvarieties} that irreducible Weil polynomials with 2-rank equal to 4 are asymptotically in bijection with the Frobenius polynomials of ordinary simple isogeny classes of abelian varieties. We apply Lemma~\ref{lem:irreducible}, and the argument follows identically as in the proof of Theorem~\ref{thm:ratio}.
\end{proof}

\begin{lemma} \label{lemma:zerojac}
    Fix a field $\F_q$ of characteristic 2. As $g\to\infty$, the proportion of isogeny classes of $g$-dimensional abelian varieties over $\F_q$ which contain the Jacobian of a hyperelliptic curve approaches 0.
\end{lemma}

\begin{proof}
    As stated in \cite[Thm 1.1]{dipippo2000realpolynomialsrootsunit}, the total number of isogeny classes grows as $\Theta(q^{g(g+1)/4})$; and as noted in Corollary~\ref{cor:num_vs}, the number of isomorphism classes of genus $g$ hyperelliptic curves over $\F_q$ grows as ${\mathcal{O}}(q^{2g-1})$. The result follows.
\end{proof}

\begin{remark}
    The procedure described in Section~\ref{sec:proof_alg} can be easily implemented in code to compute all such parity obstructions that can be proven using the elementary point counting argument. By computing these obstructions for dimension $g \leq 22$, we have observed that the proportion of the $2^g$ possible patterns of residues modulo 2 that can be proven to be obstructions to the associated isogeny class containing a hyperelliptic Jacobian rapidly approaches $100\%$. This is expected, as the proportion of isogeny classes of $g$-dimensional abelian varieties over a given field $\F_q$ which contain a hyperelliptic Jacobian approaches zero as $g\to\infty$ as in Lemma~\ref{lemma:zerojac}. In Figure~\ref{fig:genus_g} we demonstrate this pattern, where the data was generated using a cutoff of point counts up to extension fields of degree at most only $2g$. The complete set of obstructions up to genus 20 generated using this method can be found at \cite{obstructions_database}.
\end{remark}

\begin{conjecture}
    The proportion of isogeny classes of 3-dimensional abelian varieties over characteristic 2 finite fields containing a Jacobian of a hyperelliptic curve approaches $\frac{1}{2}$ as $q \to \infty$.
\end{conjecture}

This is analogous to the claim made in \cite[\S 5]{costa2020} for odd characteristic finite fields. We include both a re-generated graph to support their claim with additional data that has recently become available, in Figure~\ref{fig:odd_hyp_counts}, and an analogous graph for even $q$ to support our own claim in Figure~\ref{fig:even_hyp_counts}.

\begin{figure}[h]
    \centering
    \includegraphics[width=0.5\linewidth]{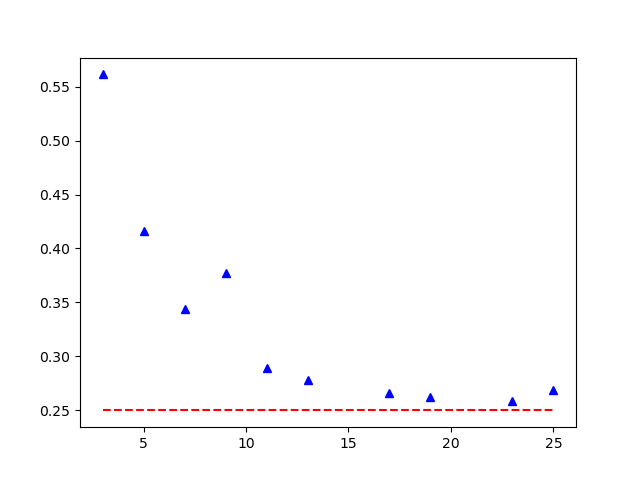}
    \caption{Ratio of isogeny classes that do not contain a Jacobian of a hyperelliptic curve versus field cardinality (odd), plotted against a conjectured asymptote of 0.25.}
    \label{fig:odd_hyp_counts}
\end{figure}

\begin{figure}[h]
    \centering
    \includegraphics[width=0.5\linewidth]{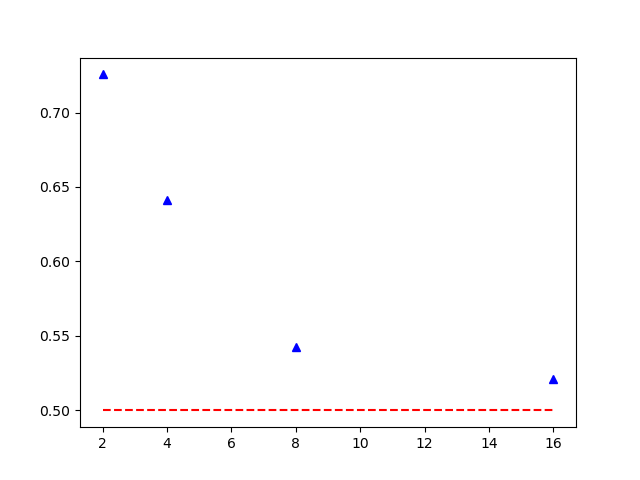}
    \caption{Ratio of isogeny classes that do not contain a Jacobian of a hyperelliptic curve versus field cardinality (even), plotted against a conjectured asymptote of 0.5.}
    \label{fig:even_hyp_counts}
\end{figure}

\begin{figure}[h] \label{fig:genus_g}
    \centering
    \includegraphics[width=0.5\linewidth]{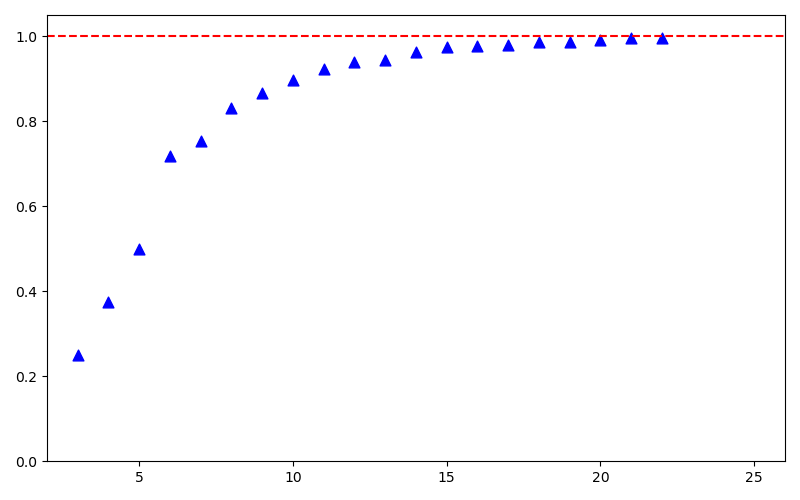}
    \caption{Lower bounds on the proportion of modulo 2 residues which prove a degree $2g$ Weil polynomial not to be the Frobenius polynomial of a genus $g$ hyperelliptic Jacobian, versus $g$, plotted against conjectured asymptote of $100\%$.}
\end{figure}

\section{Generating Hyperelliptics}\label{sec:algorithm} 

In this section we present a new algorithm for enumerating hyperelliptic curves over a finite field of characteristic 2, and discuss the runtime asymptotics. Recall that a genus $g$ hyperelliptic curve over $\F_{2^n}$ takes the form: $y^2 + v(x)y = u(x)$, where $2g+1 \leq\max(\deg(u(x)), 2\deg(v(x)))\leq2g+2$. Note also that we can always assume $v(x)$ to be monic by a simple change of coordinates. The naive approach to enumeration is to list all candidates for $u(x)$ and $v(x)$, and to verify the following criteria for each. Due to \citeauthor{xarles2020censusgenus4curves}~\cite{xarles2020censusgenus4curves}, to verify that $v(x)$ and $u(x)$ define a valid hyperelliptic curve, we only need to check the following criteria:

\begin{lemma}
    Consider the equation $y^2 + v(x)y = u(x)$ where $2g+1 \leq\max(\deg(u(x)), 2\deg(v(x)))\leq2g+2$, $v(x)$ is monic, and $u(x)$ is non-constant. This defines a hyperelliptic curve of genus $g$ over $\F_{2^n}$ if and only if $\gcd(v(x), u'(x)^2+v'(x)^2u(x))=1$ and either $deg(v(x)) = g+1$ or $a_{2g+1}^2 \neq a_{2g+2}b_g^2$, where $u(x) = \sum_{i=0}^{2g+2} a_i x^i$ and $v(x) = \sum_{i=0}^{g+1} b_i x^i$.
\end{lemma}

The brute-force enumeration considers roughly $\mathcal{O}(q^{3g+4})$ candidate pairs, where $q=2^n$. This is feasible only for $q=2$ and small genera, which motivates an optimization of this approach.

\citeauthor{xarles2020censusgenus4curves} provides the following two isomorphism invariants for hyperelliptic curves over $\F_2$, which are proven in \cite{xarles2020censusgenus4curves}. We present stronger results that hold true for all finite fields of characteristic 2, and we utilize these invariants in Algorithm ~\ref{alg:algorithm} to collapse isomorphism classes of curves. First we define the $m$-action of $\GL_2(\F_{q^n})$ on the set $R_m = \F_{2^n}[x]/(x^{m+1})$ of polynomials over $\F_{2^n}$ of degree at most $m$. Let $\psi_m(A)(f(x)) = (cx+d)^m f(\frac{ax+b}{cx+d})$ where $A = \begin{pmatrix}
    a & b \\
    c & d \\
\end{pmatrix} \in \GL_2(\F_{q^n})$ and $ v(x) \in R$. Note that this defines a group action applied on the right. 

\begin{lemma}\label{lemma:partial_action}
    Let $C_1$ and $C_2$ be isomorphic hyperelliptic curves of genus $g$ over $\F_{2^n}$ defined by $y^2 + v_i(x)y=u_i(x)$, satisfying $\max(2\deg(v_i),\deg(u_i)) = 2g+1 \text{ or } 2g+2$. There exists a matrix $A \in \GL_2(\F_{q^n})$, and a scalar $\lambda \in \F_{2^n}^\times$ satisfying $v_2(x)=\lambda\psi_{g+1}(A)(v_1(x)).$
\end{lemma}

\begin{proof}
    Due to Liu~\cite[Corr. 7.4.33]{LIU}, we have the following result: two isomorphic genus $g$ hyperelliptic curves over any field $k$ presented as above are related by the exchange of variables $(x,y) \mapsto (\frac{ax+b}{cx+d}, \frac{r(x)+\lambda y}{(cx+d)^{g+1}})$ according to some $A = \begin{pmatrix}
    a & b \\
    c & d \\
\end{pmatrix} \in \GL_2(k)$, $\lambda \in k^\times $, and $r(t) \in k[t]/(x^{g+2})$. Computing this change of variables, we get
\begin{equation*}
\begin{multlined}
    y^2 + \lambda^{-1}(cx+d)^{g+1} v_1 \left(\frac{ax+b}{cx+d}\right) y = \\
    \lambda^{-2}((c+d)^{2g+2}u_1 \left(\frac{ax+b}{cx+d}\right) +r(x)(cx+d)^{g+1}v_1\left(\frac{ax+b}{cx+d}\right)+r^2(x)).
\end{multlined}
\end{equation*}
Replacing $\lambda$ with $\lambda^{-1}$ we observe that the left hand side is $y^2+\lambda\psi_{g+1}(A)(v_1(x))$, while the right hand side defines a valid $u_2(x)$, giving us the desired result. 
\end{proof}

\begin{remark}\label{rmk:monic}
    Note that $\psi_m(\lambda A) = \lambda^m\psi_m(A)$, and, in general, nonzero elements are not expected to have $m$-th roots, so we cannot drop the scalar in the above lemma. However, we note that when $\gcd(g+1, 2^n-1) = 1$, every element has a $(g+1)$-th root. Thus, for any matrix $A \in \GL_2(k)$, we can pick a representative $A' \in \PGL_2(k) = \GL_2(k)/k^\times$ such that $\lambda = 1$. This allows us to replace $\GL_2(k)$ with $\PGL_2(k)$ to define an action on monic polynomials. This is why we restrict our algorithm to cases when $\gcd(g+1, 2^n-1) = 1$ (in particular, this is always the case for $g = 3$). 
\end{remark}

\begin{lemma}\label{lemma:full_action}
    Let $C_1$ and $C_2$ be isomorphic hyperelliptic curves of genus $g$, where $\gcd(g+1, 2^n-1)=1$, over $\F_{2^n}$ defined by $y^2 + v(x)y=u_i(x)$, satisfying $\max(2\deg(v(x)),\deg(u_i)) = 2g+1 \text{ or } 2g+2$. There exists a matrix $A \in \Stab_{v(x)} \subseteq \GL_2(\F_{q^n})$ and $r(t) \in \F_{2^n}[t]/(t^{g+2})$ such that  $u_2(x)=\psi_{2g+2}(A)(u_1(x) + r^2(x) + v(x)r(x))$, where $\Stab_{v(x)}$ is the stabilizer of $v(x)$ under the action of $\GL_2(\F_{2^n})$.
\end{lemma}

\begin{proof}
    The proof follows immediately from the computation in the previous proof, where we now note that $\psi_{g+1}(A)(v)=v$. As in the above remark, in order to ignore the scalars in this expression, we require $\gcd(g+1, 2^n-1)=1$ so that $(g+1)$-th roots exists.
\end{proof}

Utilizing these results, we make significant improvements to the best existing enumeration algorithm \cite{huang2024censusgenus6curves} due to \citeauthor{huang2024censusgenus6curves}. The enumeration of all hyperelliptics $y^2 + v(x)y = u(x)$ is computed in two steps. First, compute all possible polynomials $v(x)$, then, for each $v(x)$, compute all possible $u(x)$ modulo the relations stated above, and for each pair $v(x), u(x)$ check whether the corresponding equation is hyperelliptic. The current algorithm due to Huang et al. runs in $\Omega(2^{(3g+3)n})$ time for each computed $v(x)$, and our algorithm improves this to overall worst-case $\mathcal{O}(\text{poly}(n)2^{(2g-1)n})$ time. Additionally, the existing algorithm given in \cite{huang2024censusgenus6curves} enumerates only over $\F_2$, and we provide the necessary generalizations to arbitrary genera and almost all extensions of $\F_2$.

 \subsection*{Description of Algorithm}
 
 \begin{algorithmic}[1] \label{alg:algorithm}
     \Require Genus $g$, degree $n$.
     \Ensure $\gcd(g+1, 2^n - 1) = 1$.
     \State Define $F = \F_{2^n}$ with generator $\alpha$, $R=F[x]$, $S=R/(x^{2g+3})$, $G = \GL_2(\F_{q^n})$.
     \State Define the set $X$ of monic polynomials in $R$ with degree $\leq g+1$.
     \State Compute the orbits of $X$ under the action of $G$, where we require the result of the action to be monic (see Remark~\ref{rmk:monic} -- this is really an action of $\PGL_2(\F_{q^n})$).
     \State Choose one representative from each orbit, call this set $P_{set}$.
     \State Initialize $\texttt{Results}$, a hash table where keys will be polynomials $v(x)$ and values will be lists of polynomials $u(x)$.
     \State Let $V=\text{Span}_{\F_2}{\{\alpha^j x^i \mid j \in \{0\ldots n-1\}, i \in \{0\ldots 2g+2\}\}}$ (the $\F_2$-vector space equivalent of $S$). 
     \For{$v(x) \in P_{\text{set}}$}
        \State Let $U=\text{Span}_{\F_2}\{\alpha^jx^iv(x)+\alpha^{2j}x^{2i} \mid j\in\{0\ldots n-1t\}, i \in \{0 \ldots g+1\} \}$, (the $\F_2$-vector space of elements $r(x)v(x)+r(x)^2$).
        \State Compute the quotient space $Q = V/U$ and $\pi:V\rightarrow V/U$, the canonical map.
        \State Compute $\textbf{stab}_{v(x)} = \{A \in G \mid \psi_{g+1}(A)(v(x))=v(x)\}$, the $G$-stabilizer of $v(x).$ 
        \State Let $\texttt{initial}=\{u(x) \in V \text{ for each } [u(x)] \in Q\} $ be any set of coset representatives.
        \While{\texttt{initial}}
            \State Pop a candidate $u(x)$ from \texttt{initial}.
            \State Compute $\textbf{deg} = \max(2\deg(v(x)), \deg(u(x)))$, the degree of the curve.
            \If{not $\textbf{deg} = 2g+1 \text{ or } 2g+2$} Continue \EndIf
            \State Compute the orbit $\texttt{orb}$ of $u(x)$ under the action of $\textbf{stab}_{v(x)}$.
            \State Remove all elements of all cosets of elements of $\texttt{orb}$ from $\texttt{initial}$.
            \State Add $u(x)$ to the $\texttt{Results}[v(x)]$.
            \EndWhile
        \EndFor
        \State Populate \texttt{Hyperelliptics} with all pairs $v(x), u(x)$ in \texttt{Results} that define a valid hyperelliptic.
        \State Return \texttt{Hyperelliptics}     
 \end{algorithmic}

For an implementation of the algorithm see \url{https://github.com/bmatvey/char_2_hyperelliptics}.


\begin{theorem}
    This algorithm produces exactly the isomorphism classes of genus $g$ hyperelliptic curves over $\F_{2^n}$, where $\gcd(g+1, 2^n-1)=1$.
\end{theorem}

\begin{proof}
    It is clear that lines \textbf{1-4} correctly produce all possible monic $v(x)$ and chooses representatives modulo the relations discussed above. In particular, we pick exactly one representative from each $\GL_2(\F_{q^n})$ orbit on monic polynomials of degree at most $g+1$, so we can proceed to the next step in the algorithm which computes all possible valid hyperelliptics with a given $v(x)$. 

    We show now that the for each $v(x)$, the algorithm finds all isomorphism classes of curves involving $v(x)$, and that from each isomorphism class, exactly one representative is chosen. 
    
    Pick a $v(x)$ and start with the relation \[\psi_{2g+2}(A)(u_2(x) + r_2(x)^2+v(x)r_2(x))=\psi_{2g+2}(B)(u_1(x) + r_1(x)^2+v(x)r_1(x)),\] where $A,B \in \Stab_{v(x)}$. Using invertibility and linearity of $\psi$, we obtain 
    \begin{equation*}
    \begin{split}
        u_2(x) &=\psi_{2g+2}(BA^{-1})(u_1(x)+r_1(x)^2+v(x)r_1(x))+r_2(x)^2+v(x)r_2(x) \\
        & = \begin{multlined}[t][14cm]\psi_{2g+2}(BA^{-1})[u_1(x)+(r_1(x)+\psi_{g+1}(AB^{-1})(r_2(x)))^2 + \\ 
        \psi_{g+1}(AB^{-1})(v(x))\psi_{g+1}(AB^{-1})(r_2(x))+v(x)r_1(x)] 
        \end{multlined} \\
        &= \psi_{2g + 2}(BA^{-1})[u_1(x) + (r_1(x) + \psi_{g+1}(AB^{-1})(r_2(x)))^2 + (\psi_{g+1}(AB^{-1})(r_2(x)) + r_1(x))v(x)]
    \end{split}
    \end{equation*}
    Where we use that $AB^{-1} \in \Stab_{v(x)}$ and $\psi_{2m}(fg)=\psi_m(f)\psi_m(g)$, so rewriting $BA^{-1}=C$ and $r_1(x)+\psi_{g+1}(AB^{-1})(r_2(x))=r_*(x)$, we obtain \[u_2(x)=\psi_{2g+2}(C)(u_1(x)+r(x)^2_*+v(x)r(x)^2_*)\] which is exactly the condition in the Lemma~\ref{lemma:full_action}. 
    
    Consider the equivalence relation on $\F_{2^n}[x]/(x^{2g+3})$ defined by $f(x)\sim g(x)$ if $f(x) = g(x)+r(x)^2+v(x)r(x)$ for some $r(x)\in\F_{2^n}[x]/(x^{g+2})$. The previous paragraph shows that $u_1(x)\sim u_2(x)$ if and only if they describe isomorphic hyperelliptics with $v(x)$. 
    
    Defining the $\F_2$-vector spaces $V$ and $U$ as in the algorithm, it is clear that the quotient space $Q=V/U = (\F_{2^n}[x]/(x^{2g+3}))/\sim$. We have seen that $GL_2(\F_{2^n})$ yields a well-defined action on $V$, and we show now that $\Stab_{v(x)}$ preserves $U$, so that $\Stab_{v(x)}$ yields a well-defined action on $V/U$. To see this, note that \[\psi_{2g+2}(A)(\alpha^jx^i v(x)+\alpha^{2j}x^{2i})=\psi_{g+1}(A)(\alpha^jx^i)v(x)+\psi_{g+1}(A)(\alpha^{2j}x^{2i})\in U\] when $A \in \Stab_{v(x)}$. 
    
    By the above discussion, if two hyperelliptics with the same $v(x)$ are isomorphic, then $\pi(u_1(x))$ and $\pi(u_2(x))$ live in the same $\Stab_{v(x)}$-orbit in $V/U$. This is sufficient to show that the given algorithm chooses at most one representative from each isomorphism class, since for each coset in $V/U$, at most one representative of its $\Stab_{v(x)}$-orbit is chosen and lifted to $V$. Completeness is also clear, since each isomorphism class will be considered at least once.
\end{proof}

The following lemma bounds the size of the output. We use this bound to prove a runtime bound.

\begin{lemma}\label{lemma:counting_us}
    For each $v(x)$, the size of the space $Q = V/U$ is $2^{(g+1)n + 1}$. 
\end{lemma}

\begin{proof}
    Firstly, $V$ has size $2^{(2g+3)n}$. Now, $U$ is defined by the map $\phi: \F_{2^n}[x]/(x^{g+2}) \to U$ given by $\phi(r(x)) = r(x)v(x) + (r(x))^2$. The space $\F_{2^n}[x]/(x^{g+2})$ clearly has dimension $(g+2)n$ over $\F_2$. To find $\ker(\phi)$, we evaluate $r_1(x)v(x) + (r_1(x))^2 = r_2(x)v(x) + (r_2(x))^2$. Simple algebraic manipulation gives us $r_1(x) + r_2(x) = v(x)$ as the only solution with $r_1(x) \neq r_2(x)$, so $\ker(\phi)$ is the one dimensional vector space with $v(x)$ as its basis. Thus, we have $\dim(U) = (g+2)n - 1$, and therefore $V/U = Q$ has size $2^{(g+1)n+1}$. 
\end{proof}

To show an upper bound for runtime, we also need to show a bound on the number of candidates $v(x)$, and analyze the runtime of computing this set. We do this by bounding the size of the stabilizer $\Stab_{\PGL_2(\F_{2^n})}(v(x))$. Henceforth, let $X$ denote the set of monic polynomials of degree at most $g+1$ over $\F_{2^n}$ and $X_d$ denote the set of monic polynomials of degree exactly $d$.

\begin{proposition}\label{prop:stab_size}

Let X be the set of monic polynomials of degree at most \(g+1\) for $g \geq 2$. Then we can bound:
\[\sum_{f \in X} \#\Stab_G(f) = \mathcal{O}(q^{g+1}),\]
where $q = 2^n$. 




\end{proposition}
\begin{proof}
    First, we note that since we are acting on monic polynomials, we are interested in the group $G = \PGL_2(\F_{2^n})$. 
    
    We compute the expected stabilizer size by using the relation $\sum_{g \in G} \#\text{fix}_{X}(g) = \sum_{f \in X} \#\Stab_G(f)$, where $\text{fix}_X(g)$ is the fixed set of $g$. We bound $\sum_{f \in X}\#\Stab_{\PGL_2(\F_{2^n})}(f)$ by counting the number of matrices $A \in \GL_2(\F_{2^n})$ for which $\Psi_{g+1}(A)(f(x)) = f(x)$. We use the fact that conjugate elements fix the same number of set elements in any group action. There are four types of $\GL_2(\F_{2^n})$ conjugacy classes \cite{COOPER}.

    An important observation is that the contribution of polynomials of degree $d < g-2$ to the sum $\sum_{g \in G} \#\text{fix}_{X}(g)$ is at most $\mathcal{O}(q^3\cdot q^{g-2}) = \mathcal{O}(q^{g+1})$ assuming every such polynomial is fixed by every element of $\PGL_2(\F_{2^n})$. This is within our desired bound, so we do not need to consider polynomials of degree less than $g-2$. 

    Now, consider a generic polynomial $f(x) = x^d + \lambda_{d-1} x^{d-1} + \ldots + \lambda_0 \in X$. We analyze the conjugacy classes one by one:

    \begin{itemize}
        \item[{\rm (\textbf{I})}] Conjugacy classes with representative $A=\begin{pmatrix}
            a & 0 \\
            0 & b \\
        \end{pmatrix}$ with $a\neq b$. Each such conjugacy class has size $\O(q^2)$. We consider three cases:
        \begin{enumerate}
            \item If the degree of $f$ is $g+1$, we have $\Psi_{g+1}(A)(f(x)) = a^dx^d + b a^{d-1}x^{d-1}\lambda_{d-1} + b^2 a^{d-2}x^{d-2}\lambda_{d-2} + \ldots + \lambda_0 b^d$. Since $\gcd(d, 2^n-1) = 1$, we must have $a = 1$. Then, we immediately get $\lambda_{d-1} = 0$, $\lambda_{d-2}=0$, $\lambda_0 = 0$ from the second, third and last terms, respectively, since $b \neq 1$. Thus, this case contributes $\O(q)\times \O(q^2)\times \O(q^{g-2}) = \O(g^{g+1})$ to the total count.
            \item If the degree of $f$ is $g$, we have $\Psi_{g+1}(A)(f(x)) = ba^dx^d + b^2 a^{d-1}x^{d-1}\lambda_{d-1} + b^3 a^{d-2}x^{d-2}\lambda_{d-2} + \ldots + \lambda_0 b^{g+1}$. The first term gives us $b = a^{-g}$, so $a$ determines $b$. Since $a \neq b$, we also have $\lambda_{d-1} = 0$. Finally, if $\lambda_{d-2} \neq 0$ and $\lambda_{d-3} \neq 0$, the corresponding terms give us $a^{2g + 2} = a^{3g+3} = 1$. Together, these would imply that $a^{g+1} = 1$, meaning $a = b =1$, a contradiction. Thus, one of $\lambda_{d-2}$ or $\lambda_{d-3}$ must be 0 for any fixed point $f(x)$. Therefore, this case contributes $\O(q^2)\times \O(q) \times \O(q^{g-2}) = \O(q^{g+1})$ to the total. 
            
            Note this argument fails if $d = g = 2$, as then the $\lambda_{d-3}$ term doesn't exist. In this case, we must have $b = a^{-2}$ for $a \neq 1$ and $\lambda_0 = 0$. Thus, the total contribution is $\O(q^3)$, as desired.
            \item The remaining cases are $g-2 \leq \deg(f) < g$. We have $\Psi_{g+1}(A)(f(x)) = b^{g+1-d}a^dx^d + b^{g+2-d} a^{d-1}x^{d-1}\lambda_{d-1} + b^{g+3-d} a^{d-2}x^{d-2}\lambda_{d-2} + \ldots + \lambda_0 b^{g+1}$. The first term gives us $b^{g+1 -d} = a^{-d}$, and since $d \geq g-2$, this is at most cubic in $b$, so for every $a$ there are $\O(1)$ choices for $b$. Moreover, $a \neq b$ still gives us $\lambda_{d-1} = 0$, so the contribution to the total is $\O(q^2) \times \O(q) \times \O(q^{g-2}) = \O(q^{g+1})$. Note this argument fails is $d = 0$, but if $g = 2$ and $d = 0$, $b = 1$ and $a$ is free, so we also get a $\O(q^3)$ bound, as desired. 
        \end{enumerate}

        \item[{\rm (\textbf{II})}] Conjugacy classes with representative $A=\begin{pmatrix}
            a & 0 \\
            0 & a \\
        \end{pmatrix}$. Each conjugacy class has size 1, and there are $q-1$ such classes. Computing $\psi_{g+1}(A)(f(x)) = a^{g+1}f(x)$, we see that we must have $a = 1$. Then every $f(x)$ is a fixed point. This gives us a total contribution of $\O(q^{g+1})$ for this conjugacy class. 

        \item[{\rm (\textbf{III})}] Conjugacy classes with representative $A=\begin{pmatrix}
            a & 1 \\
            0 & a \\
        \end{pmatrix}$. Each conjugacy class has size $\O(q^2)$, and there are $q-1$ such classes. Computing $\psi_4(A)(f(x)) = a^{g+1-d}(ax+1)^d + a^{g+2-d}\lambda_{d-1}(ax+1)^{d-1} + \ldots + \lambda_0 a^{g+1}$. The $x^d$ term gives us $a^{g+1} = 1$, implying $a = 1$. With this, our requirement is now $f(x) = f(x+1) = (x+1)^d + \lambda_{d-1}(x+1)^{d-1} + \ldots + \lambda_0$. Considering the constant terms, this gives us a relationship $1 + \lambda_{d-1} + \lambda_{d-2} + \ldots + \lambda_1 = 0$, which reduces the dimensionality of the space of possible $f(x)$ by 1. Now, we look at the $x^{d-1}$ and $x^{d-2}$ terms. The first gives $dx^{d-1} + \lambda_{d-1}x^{d-1} = \lambda_{d-1}x^{d-1}$. If $d$ is odd, such an $f(x)$ cannot exist and the contribution is 0. If $d$ is even, we look at the next term: $\binom{d}{2} x^{d-2}+\lambda_{d-1}(d-1)x^{d-2} + \lambda_{d-2}x^{d-2} = \lambda_{d-2}x^{d-2}$. This implies $\binom d2 = \lambda_{d-1}(d-1)$, which determines $\lambda_{d-1}$. This brings down the dimensionality of the space of potential fixed $f(x)$ to at most $g-1$, so the total contribution of this case is $\O(q^2) \times 1 \times \O(q^{g-1}) = \O(q^{g+1})$. Note that this argument doesn't work as written for $d = 0, 1, 2$, but one can check that each of these gives a bound of at most $\O(q^3)$ using an adapted version of the above argument. 

        \item[{\rm (\textbf{IV})}] Conjugacy classes with representative $A=\begin{pmatrix}
            0 & -a_0 \\
            1 & a_1 \\
        \end{pmatrix}$, where $x^2 - a_1x + a_0$ is irreducible. Each conjugacy class has size $\O(q^2)$, and there are $\O(q^2)$ such classes. We have $\Psi_{g+1}(A)f(x) = a_0^d + \lambda_{d-1}(x+a_1)a_0^{d-1} +\ldots+\lambda_0(x+a_1)^{g+1}$. Note that this immediately implies $d = g+1$. This is more conveniently rewritten as $\lambda_0(x+a_1)^d + \lambda_1(x+a_1)^{d-1}a_0 + \lambda_2(x+a_1)^{d-2}a_0^2 + \ldots + a_0^d$. Matching this with $f(x)$, we get the following equations:
        \begin{enumerate}
            \item $\lambda_0 = 1$
            \item $d a_1 + a_0 \lambda_1 = \lambda_{d-1}$
            \item $\binom d2 a_1^2 + \lambda_1(d-1)a_0 a_1 + \lambda_2 a_0^2 = \lambda_{d-2}$
            \item $a_1^d + \lambda_1 a_1^{d-1}a_0 + \ldots + \lambda_{d-1}a_1 a_0^{d-1} + a_0^d = 1$
        \end{enumerate}
        If $g \geq 5$, these clearly give four linearly independent conditions, reducing the number of potential $f(x)$ to $\O(q^{g-3})$. Thus, this case contributes $\O(q^2) \times \O(q^2) \times \O(q^{g-3}) = \O(q^{g+1})$.

        Once more, there are a few low-degree cases that require special consideration. 
        \begin{enumerate}
            \item If $d = 3$, manually solving the above equations shows that the first three equations determine $\lambda_0, \lambda_1, \lambda_2$ in terms of $a_0, a_1$, and the fourth gives a condition on $a_0, a_1$. In particular, once $a_0$ is determined, it becomes a quartic in $a_0$, and thus for every $a_0$, there are $\O(1)$ choices for $a_1$ which have a nonempty set of fixed points. This gives the required bound of $\O(q^3)$. 
            \item If $d = 4$, one can solve the above system manually. It turns out that $a_0, a_1$ uniquely determine $f(x)$ in most cases. There are $\O(1)$ cases when $f(x)$ remains a two parameter family and $\O(q)$ cases when $f(x)$ remains a 1 parameter family, but the sum is always bounded by $\O(q^2)\times \O(q^2) = \O(q^4)$, as desired. 
            \item If $d = 5$, we proceed just as in the $d = 4$ case. Solving the system generally yields a 1-parameter family of $f(x)$, except for $\O(q)$ specific choices of $a_1, a_2$, for which $f(x)$ remains a 2 parameter family. Again, the sum is still bounded by $\O(q^3)\times \O(q^2) = \O(q^5)$, the desired bound. 
        \end{enumerate}
    \end{itemize}

     Since there are $\O(1)$ cases considered above (note that for each $g$, we consider up to 4 distinct degrees $d$), adding together the bounds for all of the terms described above gives us $\sum_{g \in \PGL_2(\F_{2^n})}\#\text{fix}_X(g) =\O(q^{g+1})$, our desired bound.

\end{proof}

\begin{corollary}\label{cor:burnside}
    The number of candidates $v(x)$ over $\F_{2^n}$ is asymptotically $\mathcal{O}(2^{(g-2)n})$.
\end{corollary}

\begin{proof}
    We apply Burnside's Lemma: $\#(\PGL_2(\F_{q^n})$-$\text{orbits of } X) := |X/G| = \frac{1}{|G|}\sum_{x \in X} |\Stab_G(x)|$. Using Proposition~\ref{prop:stab_size}, we have
    \begin{equation*}
    \begin{split}
    |X/G|
    &= \frac{2^n-1}{(2^{2n}-1)(2^{2n}-2^n)} (\O(2^{n(g+1)}) \\
    &= \mathcal{O}(2^{n(g-2)}).
    \end{split}
    \end{equation*}

\end{proof}

\begin{corollary}\label{cor:num_vs}
    This immediately gives us a bound of $\mathcal{O}(2^{(2g-1)n})$ on the number of isomorphism classes of hyperelliptics over $\F_{2^n}$ for a genus $g$. 
\end{corollary}

\begin{proof}
    Combining Lemma~\ref{lemma:counting_us} with Corollary~\ref{cor:num_vs}, the result follows.
\end{proof}

\begin{theorem}
    The provided algorithm runs in time $\mathcal{O}(\text{poly}(n)2^{(2g-1)n})$ expected and $\mathcal{O}(\text{poly}(n)2^{(2g+2)n})$ worst-case.
\end{theorem}

\begin{proof}
    Computing the set $P_{set}$ makes use of GAP's \cite{GAP} \texttt{OrbitsDomain()} function. Computing each orbit has cost $\textit{size of orbit} \times \textit{number of generators} \times \textit{cost of group action}$. Since $\GL_2(\F_{q^n})$ can be minimally generated by three generators, and the cost of computing the group action is polynomial in $n$, the total cost of this step is \[\mathcal{O}(1)\times\text{poly}(n)\times\frac{1}{2^n-1}\#(\F_{2^n}[x]/(x^{g+2})) = \text{poly}(n)\times2^{(g+1)n}.\] We see that this is negligible compared to the second stage of the algorithm, so it can be ignored. 

    For the expensive step of the computation, Corollary~\ref{cor:num_vs} gives us that the number of $v(x)$ we need to iterate through is bounded by $\O(2^{n(g-2)})$. For each of these, we must iterate through $\O(2^{(g+1)n})$ possible $u(x)$. At each step, we have to iterate through $\Stab_G(v(x))$, which has expected size $\O(1)$ (by Proposition~\ref{prop:stab_size}), but has size $\O(q^3)$ in the worst case. This gives the claimed runtimes. 
\end{proof}


\begin{remark}
    It is not possible to make use of GAP's \cite{GAP} orbit computation functionality in Lines \textbf{17-19}, since the set \texttt{initial} is not necessarily closed under the group action.
\end{remark}

\begin{remark}
    For genus 3, we ran our implementation on an AMD Opteron 6176 for small finite fields. The runtimes are listed below (assuming a single core is being used):
    \begin{itemize}
        \item $\F_2$: ~40 CPU-seconds
        \item $\F_4$: ~120 CPU-seconds
        \item $\F_8$: ~ 30 CPU-minutes
        \item $\F_{16}$: ~15 CPU-hours
        \item $\F_{32}$: ~650 CPU-hours
    \end{itemize}
    The small fields are likely significantly impacted by computational overhead (such as launching sage). Our implementation is highly parallelized, so the real-world time can be reduced significantly.
\end{remark}





\newpage

\nocite{tange_2025_15717374}

\pagebreak
\printbibliography

\end{document}